\documentclass{amsart}
\usepackage{amssymb}

\newtheorem{thm}{Theorem}[section]

\newtheorem{cor}{Corollary}[section]

\newcommand{\essup}{\rm essup}

\usepackage{graphicx,psfrag}
\numberwithin{equation}{section}

\def\k0{\kappa_0}

\def\bfu{{\bf{u}}}
\def\bfx{{\bf{x}}}
\def\bff{{\bf{f}}}
\def\bfg{{\bf{g}}}
\def\bfv{{\bf{v}}}
\def\bfw{{\bf{w}}}

\begin{document}
\title[A strengthening of the energy inequality]
{A strengthening of the energy inequality for the Leray-Hopf solutions
of the 3D periodic Navier-Stokes equations}
\author{R. Dascaliuc}
\address{Department of Mathematics\\
University of Virginia\\ Charlottesville, VA 22904}
\date{\today}
\begin{abstract}
In present note we establish the following inequality for the the
Leray-Hopf solutions of the 3-D $\Omega$-periodic Navier-Stokes Equations:
\[\phi(|u(t)|^2)-\phi(|u(t_0)|^2)\le
2\int_{t_0}^{t}\phi'(|u(\tau)|^2)
\left[-\nu|A^{1/2}u(\tau)|^2+(g(\tau),u(\tau))\right]\,d\tau\]
for all $t_0$ Leray-Hopf points, $t\ge t_0$, and
$\phi:\mathbb{R}_{+}\to\mathbb{R}$ is an absolutely continouos non-decreasing
function with bounded derivative.
Here $(\cdot,\cdot)$ and $|\cdot|$
is correspondingly the $L^2$ inner product and the $L^2$ norm on $\Omega$,
and $A$ is the Stokes operator.
\end{abstract}
\maketitle


\section{preliminaries}
We consider three dimensional incompressible Navier-Stokes equations (NSE)
\begin{equation}\label{inc-nse}
\begin{aligned}
\frac{\partial}{\partial t}\bfu(t,\bfx)-\nu\Delta \bfu(t,\bfx)
+(\bfu(t,\bfx)\cdot\nabla)\bfu(t,\bfx)+\nabla p(t,\bfx)&=\bff(t,\bfx)\\
\nabla\cdot\bfu(t,\bfx)&=0,
\end{aligned}
\end{equation}
where $\bfx\in\mathbb{R}^3$, $t\in\mathbb{R}$; and
$\bfu(t,\bfx),\bff(t,\bfx)\in\mathbb{R}^3$ and
$p(t,\bfx)\in\mathbb{R}$ for all $t$ and $\bfx$.
We supplement \eqref{inc-nse} with periodic boundary conditions
\begin{equation}\label{bc-nse}
\begin{aligned}
\bfu,p,\bff\ &\mbox{are}\ \Omega-\mbox{periodic},\\
\int_{\Omega}\bfu(t,\bfx)\,d\bfx=0\ &\mbox{and}\
\int_{\Omega}\bff(t,\bfx)\,d\bfx=0\ \mbox{for all}\ t,
\end{aligned}
\end{equation}
where
\[\Omega=[0,L]^3,\]
and the initial condition
\begin{equation}\label{ic-nse}
\begin{aligned}
\bfu(0,\bfx)=\bfu_0(\bfx)\ \mbox{for all}\ \bfx,\\
\bfu_0\ \mbox{satisfies \eqref{bc-nse}}.
\end{aligned}
\end{equation}

Taking the Leray projector $P_L$ of \eqref{inc-nse} we obtain
\begin{equation}\label{nse}
\frac{d}{dt}\bfu+\nu A\bfu+B(\bfu,\bfu)=\bfg,
\end{equation}
where
\begin{align*}
A\bfu=-P_L\Delta\bfu,\ B(\bfu,\bfu)=P_L(\bfu\cdot\nabla)\bfu,\
\mbox{and}\ \bfg=P_L\bff.
\end{align*}
We introduce the following functional space
\begin{equation}\label{H-def}
H=\left\{\bfu: \bfu\ -\ \Omega-\mbox{periodic},\ \bfu\in L^2(\Omega)^3,\
\nabla\cdot{\bfu}=0, \int_{\Omega}\bfu=0\right\}
\end{equation}
For any $\bfu,\bfv\in H$ denote
\[(\bfu,\bfv)=\int_{\Omega}\bfu\cdot\bfv\]
and
\[|\bfu|^2=\int_{\Omega}\bfu\cdot\bfu,\]
the $L^2$ inner product and norm on $H$.
Note that $A:D(A)\in H\to H$ is a positive self adjoint operator with
a compact inverse; its domain $D(A)$ is dense in $H$. Denote by
$\lambda_k$, $k\in\mathbb{N}$,
its eigenvalues arranged in the increasing order and
counting the multiplicities.
We also introduce the functional space
\begin{equation}\label{V-def}
V=\left\{\bfu\in H:\ \bfu\in H^1(\Omega)^3\right\}=
\left\{\bfu\in H:\ \bfu\in D(A^{1/2})\right\}.
\end{equation}
Define the inner product and the norm on $V$ by
\[((\bfu,\bfv))=(A^{1/2}\bfu,A^{1/2}\bfv)\]
and
\[||\bfu||^2=|A^{1/2}\bfu|^2\]
for all $\bfu,\bfv\in V$.
Note that
\[\lambda_1|\bfu|^2\le||\bfu||^2\]
for all $\bfu\in V$.

The bilinear operator $B$ has the following orthogonality property:
\[(B(\bfu,\bfv),\bfw)=-(B(\bfu,\bfw),\bfv),\ \mbox{for all}\
\bfu,\bfv,\bfw\in V,\]
and consequently
\[(B(\bfu,\bfv),\bfv)=0,\ \mbox{for all}\ \bfu,\bfv\in V.\]

In this paper we study weak solutions of the NSE, i.e. $H$-valued
functions $\bfu(t)$ that satisfy
\[
\frac{d}{dt}(\bfu(t),\bfv)+\nu((\bfu(t),\bfv))+(B(\bfu(t),\bfu(t)),\bfv)=
(\bfg(t),\bfv)\ \mbox{for all}\ \bfv\in V.
\]

For more background on the Leray-Hopf solutions of Navier-Stokes equations and
on the functional setting used here consult 
\cite{MR972259,MR1855030,MR1325465,MR1846644,MR1318914,xx,yy}

The principal result about the existence of the weak (Leray-Hopf) solutions
of the NSE can be stated as follows.
\begin{thm}\label{LH-soln}
Let $\bfu_0\in H$ and $\bfg\in L^2([0,T],H)$, where $T>0$ given.
Then there exist a weak solution $\bfu\in L^2([0,T],V)$, which is weakly
continuous as a function from $[0,T]$ to $H$.

Moreover, the following inequality holds
\begin{equation}\label{LH-ineq}
|\bfu(t)|^2-|\bfu(t_0)|^2\le
2\int_{t_0}^{t}\left[-\nu||\bfu(\tau)||^2+(\bfg(\tau),\bfu(\tau))\right]\,d\tau
\end{equation}
for all $t_0\ge 0$ Lebesgue points of $|\bfu(\tau)|^2$ and all $t\in[t_0,T]$.
\end{thm}

In fact, the set of Lebesgue points of $|\bfu(\tau)|^2$ is dense in $[0,T]$ and
includes $t_0=0$ and any $t_0$ such that $\bfu(t_0)\in V$.

Under the additional assumptions on the initial data, one gets local
existence of the regular solutions of the NSE.

\begin{thm}\label{reg-soln}
Suppose $\bfu_0\in V$ and $\bfg\in L^2([0,T],H)$. Then there exists
$T_{*}=T_{*}(||u_0||,\nu,\Omega)>0$ such that the weak solution $\bfu$
from Theorem
\ref{LH-soln} is unique on $[0 ,T_{*}]$ and
$\bfu\in L^2([0,T_{*}], D(A))\cap C([0,T_{*}],V)$
(i.e. $\bfu$ is {\em regular} on $[0,T_{*}]$). Moreover, the inequality
\eqref{LH-ineq} becomes equality for all $t_0,t\in[0,T_{*}]$, $t\ge t_0$.
\end{thm}

It is important to mention that any weak solution of the NSE
is regular on the set
\[\mathcal{G}=\cup_n I_n,\]
where $\{I_n\}_n$ is a countable family of disjoint intervals
with cl$(\mathcal{G})=[0,T]$, and the fractal dimension of
$[0,T]\backslash\mathcal{G}$
less or equal to $1/2$.

\section{Proof of the main result}
\begin{thm}\label{thm-main}
Let $\bfg\in L^{2}([0,T],H)$ and
suppose $\bfu$ is a weak solution for the NSE.
Let $\phi:\mathbb{R}_{+}\to\mathbb{R}$ be
an absolutely continouos non-decreasing function with
$\phi'\in L^{\infty}(\mathbb{R}_{+},\mathbb{R}_{+})$.
Then the following inequality holds
\begin{equation}\label{gen-LHI}
\phi(|\bfu(t)|^2)-\phi(|\bfu(t_0)|^2)\le
2\int_{t_0}^{t}\phi'(|\bfu(\tau)|^2)
\left[-\nu||\bfu(\tau)||^2+(\bfg(\tau),\bfu(\tau))\right]\,d\tau
\end{equation}
for all $t_0\in[0,T]$ - Lebesgue point for $|\bfu(\tau)|^2$ and all
$t\in[0,T]$, $t\ge t_0$.
\end{thm}

\begin{proof}
Let $t_0,t\in[0,T]$, $t\ge t_0$, $t_0$ -  Lebesgue point for $|\bfu(\tau)|^2$.
Let
\[M={\essup}_{\tau\in[t_0,t]}\phi'(|\bfu(\tau)|^2)\]
Choose an $\epsilon>0$.
Then there exists $\delta=\delta(\epsilon)>0$ such that,
for any countable family of disjoint intervals
$\{[t_n,\tau_n]\}\subset[t_0,t]$ such that $\sum_n(\tau_n-t_n)<\delta$,
\begin{equation}\label{e-est1}
M\int\limits_{\cup_n[t_n,\tau_n]}
\frac{|\bfg(\tau)|^2}{\nu\lambda_1}d\tau
<\epsilon.
\end{equation}
and
\begin{equation}\label{e-est2}
\left|\,\int\limits_{\cup_n[t_n,\tau_n]}
\phi'(|\bfu(\tau)|^2)
\left[-\nu||\bfu(\tau)||^2+(\bfg(\tau),\bfu(\tau))\right]\,d\tau\right|
<\epsilon
\end{equation}

Recall that $\bfu$ is regular on a set
\[\mathcal{G}=\cup_n I_n,\]
where int$(I_n)=(\alpha_n,\beta_n)$ and
$\sum_n(\beta_n-\alpha_n)=t-t_0$.
Clearly, inside each of the intervals $(\alpha_n,\beta_n)$ we have
\[
\phi(|\bfu(s_1)|^2)-\phi(|\bfu(s_0)|^2)=
2\int_{s_0}^{s_1}\phi'(|\bfu(\tau)|^2)
\left[-\nu||\bfu(\tau)||^2+(\bfg(\tau),\bfu(\tau))\right]\,d\tau
\]
for all $s_0,s_1\in(\alpha_n,\beta_n)$, $s_0\le s_1$.

Note that there exists $N=N(\delta)$ such that
\[\sum\limits_{n=1}^{N}(\beta_n-\alpha_n)\ge t-t_0-\delta/2.\]
Re-arrange these first $N$ intervals such that $\beta_n\le\alpha_{n+1}$
for all $n=1..N-1$. Choose $\tau_{n-1},t_n\in(\alpha_n,\beta_n)$,
$n=1..N$ and $\tau_N=t$, so that $\tau_{n-1}<t_n<\tau_n$ for $n=1..N$
and
\[\sum\limits_{n=0}^{N}(\tau_n-t_n)<\delta.\]

Then, by the \eqref{LH-ineq},
\begin{equation*}
|\bfu(\tau_n)|^2-|\bfu(t_n)|^2\le 2\int_{t_n}^{\tau_n}
\left[-\nu||\bfu(\tau)||^2+(\bfg(\tau),\bfu(\tau))\right]\,d\tau,
\end{equation*}
for all $n=0..N$.
Then,
\begin{equation}\label{u-est}
|\bfu(\tau_n)|^2-
|\bfu(t_n)|^2\le|\bfu(\tau_n)|^2-
|\bfu(t_n)|^2+\nu \int_{t_n}^{\tau_n}
||\bfu(\tau)||^2d\tau\le
\int_{t_n}^{\tau_n}\frac{|\bfg(\tau)|^2}{\nu\lambda_1}\,d\tau,
\end{equation}
for all $n=0..N$.
On the other hand, since $\phi$ is increasing,
\begin{equation*}
\phi(|\bfu(\tau_n)|^2)-\phi(|\bfu(t_n)|^2)\le
\left\{
\begin{aligned}
&0,\ \mbox{if}\ |\bfu(\tau_n)|<|\bfu(t_n)|;\\
&M(|\bfu(\tau_n)|^2-|\bfu(t_n)|^2),\ \mbox{otherwise}.
\end{aligned}
\right.
\end{equation*}
Use \eqref{u-est} to obtain
\begin{align*}
\phi(|\bfu(\tau_n)|^2)-\phi(|\bfu(t_n)|^2)
\le
M\int_{t_n}^{\tau_n}
\frac{|\bfg(\tau)|^2}{\nu\lambda_1}d\tau.
\end{align*}
Thus, taking into the account \eqref{e-est1} and \eqref{e-est2},
we can estimate
\begin{align}\label{phi-est}
&\phi(|\bfu(t)|^2)-\phi(|\bfu(t_0)|^2)\\
&=
\sum\limits_{n=0}^{N}
\left(\phi(|\bfu(\tau_n)|^2)-\phi(|\bfu(t_n)|^2)\right)
+\sum\limits_{n=1}^{N}
\left(\phi(|\bfu(t_n)|^2)-\phi(|\bfu(\tau_{n-1})|^2)\right)
\nonumber\\
&\le M\int\limits_{\cup_n[t_n,\tau_n]}
\frac{|\bfg(\tau)|^2}{\nu\lambda_1}d\tau
+
2\sum\limits_{n=1}^{N}\int_{\tau_{n-1}}^{t_n}
\Psi(\tau)d\tau
\nonumber\\
&\le
\epsilon+\sum\limits_{n=1}^{N}\int_{\tau_{n-1}}^{t_n}
\Psi(\tau)d\tau
\le
2\epsilon+2\int_{t_0}^{t}
\Psi(\tau)d\tau\nonumber
\end{align}
where
\[\Psi(\tau)=\phi'(|\bfu(\tau)|^2)
\left[-\nu||\bfu(\tau)||^2+(\bfg(\tau),\bfu(\tau))\right].\]
Returning to \eqref{phi-est}, make $\epsilon\to0$ to obtain \eqref{gen-LHI}.

\end{proof}

\begin{cor}
Under the assumptions of Theorem \ref{thm-main}, let
$\psi\in C^1(\mathbb{R}_{+},\mathbb{R})$
be such that $\psi'(\xi)\le 0$ for all $\xi\ge0$.
Then the following inequality holds
\begin{equation}\label{gen-LHI-negative}
\psi(|\bfu(t)|^2)-\psi(|\bfu(t_0)|^2)\ge
2\int_{t_0}^{t}\psi'(|\bfu(\tau)|^2)
\left[-\nu||\bfu(\tau)||^2+(\bfg(\tau),\bfu(\tau))\right]\,d\tau
\end{equation}
for all $t_0\in[0,T]$ - Lebesgue point for $|\bfu(\tau)|^2$ and all
$t\in[0,T]$, $t\ge t_0$.
\end{cor}

\begin{proof}
This result follows by applying Theorem \ref{thm-main} to $\phi=-\psi$.
\end{proof}

\section*{acknowledgements} The author would like to thank Professor Ciprian Foias for helpful suggestions and comments.


\bibliographystyle{plain}
\bibliography{3Dnse}

\end{document}